\documentclass[11pt]{article}

\usepackage{hyperref}
\usepackage{times}  % Fonts
\usepackage{mathpazo}
\usepackage{amssymb,amsmath,amsthm}
\usepackage{epsfig}

\usepackage[table]{xcolor}
\usepackage{caption}
\usepackage{graphicx}

\newtheorem{theorem}{Theorem}

\newtheorem{definition}[theorem]{Definition}

\newtheorem{lemma}[theorem]{Lemma}

\newtheorem{corollary}[theorem]{Corollary}

\newcommand{\qedsymb}{\hfill{\rule{2mm}{2mm}}}
\renewenvironment{proof}[1][]{\begin{trivlist}
\item[\hspace{\labelsep}{\bf\noindent Proof#1:\/}] }{\qedsymb\end{trivlist}}

\def\R{\mathbb{R}}

\def\Q{\mathbb{Q}}

\newcommand{\rank}{\mathop{\mathrm{rank}}}
\newcommand{\Rbool}{{\rank}_\mathbb{B}}
\newcommand{\Rbin}{{\rank}_{\mathrm{bin}}}
\newcommand{\Rreal}{{\rank}_\mathbb{R}}

\providecommand{\keywords}[1]{\textbf{\textit{Key Words---}} #1}

\begin{document}

\title{A Study of the Binary and Boolean Rank
of Matrices with Small Constant Real Rank}

\author{
Michal Parnas\thanks{School of Computer Science, The Academic College of Tel Aviv-Yaffo, Tel Aviv 61083, Israel. Email address: {\tt michalp@mta.ac.il}}
\and
Adi Shraibman\thanks{School of Computer Science, The Academic College of Tel Aviv-Yaffo, Tel Aviv 61083, Israel. Email address: {\tt adish@mta.ac.il}}
}

\maketitle

\begin{abstract}
We initiate the study of the binary and Boolean rank of $0,1$ matrices that have a small rank over the reals.
The relationship between these three rank functions is an important open question, and here we prove
that when the real rank $d$ is a small constant,  the gap between the real and the binary and Boolean rank is a small constant.
We give tight upper and lower bounds on the Boolean and binary rank of matrices with real rank $1 \leq d \leq 4$,
as well as determine the size of the largest isolation set in each case.
Furthermore, we prove that for $d = 3,4$, the circulant matrix defined by a row with $d-1$ consecutive ones followed by $d-1$ zeros,
is the only matrix of size $(2d-2)\times (2d-2)$ with real rank $d$ and Boolean and binary rank and isolation set of size $2d-2$,
and this matrix achieves the maximal gap possible between the real and the binary and Boolean rank for these values of $d$.

Our results can also be interpreted in other equivalent terms, such as finding the minimum number of bicliques needed to partition or cover
the edges of a bipartite graph whose reduced adjacency matrix has real rank $1 \leq d \leq 4$.
We use a combination of combinatorial and algebraic techniques combined with the assistance of a computer program.

%\keywords{First keyword  \and Second keyword \and Another keyword.}


\end{abstract}

\section{Introduction}

The real (standard) rank of an $n\times m$ matrix $M$ over $\R$, denoted by $\Rreal(M)$, is an important concept in many applications and fields of mathematics and computer science
and has several equivalent definitions. The following definition is used here:
the real rank of $M$ is the minimal $d$ for which there exist real matrices $A$ and $B$ of size $n \times d$ and $d \times m$ respectively,
such that $M = A \cdot B$, where the operations are over $\R$.
In a similar way, consider the following rank functions defined over a semi-ring (see e.g. Gregory and Pullman~\cite{GregoryPullman}).
The {\em binary rank} of a $0,1$ matrix $M$ of size $n\times m$, denoted by $\Rbin(M)$, is
the minimal $d$ for which there exist $0,1$ matrices $A$ and $B$ of size $n \times d$ and $d \times m$ respectively, such that $M = A \cdot B$, where the operations are over the integers.
The {\em Boolean rank} of a $0,1$ matrix  $M$, denoted by $\Rbool(M)$, is defined similarly to the binary rank, but here the operations are under Boolean arithmetic
(namely, $0+x=x+0=x$, $1+1=1 \cdot 1 = 1$, and $x \cdot 0 = 0 \cdot x = 0$).
Such decompositions $M=A \cdot B$ are called {\em optimal} for the given rank function.
The broad interest in the Boolean and binary rank stems also from the following equivalent definitions,
and we will make use of these alternative definitions throughout the paper.

A {\em monochromatic combinatorial rectangle} in a $0,1$ matrix $M$ is a sub-matrix of $M$, all of whose entries have the same value.
The {\em partition number} of $M$ is the minimum number of monochromatic rectangles required to partition the ones of $M$, and is equal to  $\Rbin(M)$.
The {\em cover number} of $M$ is the minimum number of monochromatic rectangles required to cover the ones of $M$, where rectangles may overlap,
and  is equal to  $\Rbool(M)$ (see Gregory,  Pullman,  Jones, and  Lundgren~\cite{Gregory}).

The binary and Boolean rank  are also related to concepts in graph theory.
Given a bipartite graph $G$, the {\em biclique partition number}, $bp(G)$, and  {\em biclique cover number},  $bc(G)$,  are equal to the minimum
number of bicliques needed to  partition or cover, respectively, all edges of $G$.
Let $M$ be the reduced  adjacency matrix of $G$, that is, if $G$ has $n$ and $m$ vertices on each side, then $M$ has $n$ rows and $m$ columns
and $M_{i,j} = 1$  if and only if $(i,j)$ is an edge of  $G$, where $M_{i,j}$ is the element on the $i$'th row and $j$'th column of  $M$.
A monochromatic combinatorial rectangle of ones in $M$ corresponds to a biclique in $G$,
and, thus, $bp(G) = \Rbin(M)$ and $bc(G) = \Rbool(M)$  (see~\cite{Gregory}).

The binary and Boolean rank are also closely related to the field of communication complexity.
Here two players are given a joint matrix $M$, where one player receives a row index $i$ and the second a column index $j$,
and their goal is to determine  $M_{i,j}$ while minimizing the number of bits communicated between them in the worst case.
It was shown that $\log \Rbin(M)$  approximates the deterministic communication complexity of $M$ up to a polynomial, and determines
exactly the unambiguous nondeterministic communication complexity of $M$,
whereas $\log \Rbool(M)$ determines exactly the non-deterministic communication complexity of $M$.
Furthermore, the famous  log-rank conjecture of Lov{\'a}sz and Saks~\cite{LovaszS88} suggests that the deterministic communication complexity is polynomially related to $\log \Rreal(M)$
(see, e.g., Lovett~\cite{LovettA14}). Since the log-rank conjecture seems so far unreachable, any further understanding of the properties of the Boolean and binary rank in comparison with the real rank
can assist.

Computing the Boolean and binary rank is $NP$-complete (See  Orlin~\cite{orlin1977contentment} and Jiang and Ravikumar~\cite{jiang1993minimal}).
Chalermsook,  Heydrich,  Holm, and  Karrenbauer showed in~\cite{chalermsook2014nearly}
that it is hard in general to approximate the Boolean rank
of a matrix $M$ of size $n \times n$  to within a factor of $n^{1 - \epsilon}$ for any given $\epsilon > 0$.
Chandran, Issac and  Karrenbauer~\cite{chandran2017parameterized}
give a polynomial time algorithm with an approximation ratio of $n/\log n$ for the Boolean and binary rank.

In the framework of parameterized algorithms, Fleischner,  Mujuni, Paulusma, and  Szeider showed in~\cite{fleischner2009covering}
that the Boolean and binary rank are fixed-parameter tractable with running time $f(d)poly(n,m)$, where $d$ is the rank in question.
Gramm, Guo, H\"{u}ffner, and  Niedermeier~\cite{gramm2009data} give an exact solution for the Boolean rank with $f(d) = 2^{2^{O(d)}}$, and
Chandran, Issac and  Karrenbauer  provide in~\cite{chandran2017parameterized} an algorithm for the binary rank $d$ with $f(d) = 2^{O(2^d)}$.
Moreover,~\cite{chandran2017parameterized} prove that there is no algorithm for the Boolean rank with less than a double exponential complexity in $d$
unless the Exponential Time Hypothesis (ETH) is false, thus matching up to polynomial factors the algorithm of~\cite{gramm2009data}.
%See Downey and Fellows in~\cite{downey1995fixed} for a detailed description of parameterized complexity.

There are simple bounds on the relationship between the real rank and the binary and Boolean rank.
It always holds that $\Rreal(M) \leq \Rbin(M)$, whereas the Boolean rank can be smaller or larger than the real rank and is always bounded above by the binary rank.
The identity matrix is a trivial example where all three ranks are equal,
and the binary and Boolean rank are always equal to the real rank for real rank $1$ or $2$.
When $\Rreal(M) \geq 3$ there are examples of matrices for which all three rank functions differ (see Section~\ref{prliminaries}).

A useful tool for proving lower bounds on the binary and Boolean rank of $M$ is finding a large {\em isolation set} in $M$
(or a {\em fooling set} as it is called in communication complexity),
which is a subset of $1$ entries of $M$, such that no two of these ones belong to an all-one $2\times 2$ sub-matrix of $M$,
and no two of these ones are in the same row or column of $M$.
Thus, no two ones of an isolation set can belong to the same monochromatic rectangle,
and the size of the largest isolation set of $M$, denoted by $i(M)$, provides a lower bound on $\Rbin(M)$ and $\Rbool(M)$ (see, e.g, Beasley~\cite{BEASLEY20123469}).
However, there can be a large gap between $i(M)$ and $\Rbool(M), \Rbin(M)$.
%(see e.g., de Caen, Gregory, and Pullman~\cite{Caen2} and Haviv and Parnas~\cite{haviv2023binary}).
For example, the matrix $D_n$ of size $n \times n$, $n \geq 3$, which has an all-zero main diagonal and ones elsewhere,
has $i(D_n) = 3$ whereas $\Rbool(D_n) = \Theta(\log_2 n)$ and $\Rreal(D_n) = \Rbin(D_n)= n$ (see de Caen, Gregory, and Pullman~\cite{Caen2} and Haviv and Parnas~\cite{haviv2023binary}).

\subsection{Our Results}
Our goal is to find the largest gap possible between the real rank $d$ and the Boolean and binary rank of a $0,1$ matrix $M$, for small constant values of $d$,
and to characterize these matrices from a combinatorial and algebraic point of view,
as well as understand the size of the largest isolation set possible.
It is enough to consider only {\em basic} matrices, that is, matrices that do not have all-zero rows and columns, and do not have identical rows or columns.

Using a result of de Caen, Gregory, Henson, Lundgren, and Maybee~\cite{realnonnegative} it is possible to prove the following theorem and get a multiplicative gap
of almost $2$ between the real and the binary and Boolean rank as the size of the matrix grows.
The matrices used for the proof of this theorem also have large isolation sets which are equal to the binary and Boolean rank of the matrix.

\begin{theorem}
\label{theodoublegap}
For all $n \geq 4$, $n \neq 5$, there exists a basic $0,1$ matrix $M$ of size $n \times n$, with $\Rreal(M) = \lfloor n/2\rfloor + 1$,
$\Rbin(M) = \Rbool(M) = i(M) = 2\lfloor n/2\rfloor$.
\end{theorem}

The even sized matrices used in the proof of Theorem~\ref{theodoublegap} are the circulant matrices $C_n$ of size $n \times n$ defined by a row with $n/2$ consecutive ones followed by $n/2$ zeros.
The odd sized matrices are constructed by taking  an even sized matrix $C_n$ and adding to it a carefully selected row and column
which are a linear combination of the rows and columns of $C_n$.
Note that the theorem does not hold for $n = 5$, and as we prove,
there does not exist a basic $0,1$ matrix of size $5 \times 5$ and real rank $\lfloor 5/2\rfloor + 1 = 3$ (see Lemma~\ref{lemma3}, Section~\ref{Sec3}).

By taking any of the matrices $M$ given in Theorem~\ref{theodoublegap} and computing the Kronecker product of $M$ with itself $s$ times,
it is possible to amplify the gap and get a matrix $M^{\otimes s}$ with real rank $(\lfloor n/2\rfloor + 1)^s$ and
binary rank $(2\lfloor n/2\rfloor)^{s}$ since $i(M) =  2\lfloor n/2\rfloor$.
See also Dietzfelbinger,  Hromkovi{\v{c}} and Schnitger~\cite{dietzfelbinger1996comparison} who proved a similar result for  $n = 4$.
Friesen, Hamed,  Lee and Theis show in~\cite{friesen2015fooling} that for any $d \geq 1$ there exists a matrix $M_d$ over $\Q$,
such that $\Rreal(M_d) = d$ and $M_d$ has an isolation set of size $\binom{d+1}{2}$.
Shigeta and  Amano~\cite{shigeta2015ordered} construct $0,1$ matrices of size $n \times n$ with real rank $n^{1/2 + o(1)}$ and an isolation set of size $n$.

It is well known that the number of distinct rows and columns of a $0,1$ matrix $M$ is bounded above by $2^{\Rreal(M)}$ (see, for example, Hrube\v{s}~\cite{hrubevs2024hard}).
This gives a trivial upper bound of $2^{\Rreal(M)}$ on the Boolean and binary rank of a matrix as a function of the real rank.
But for matrices with small constant real rank we prove that the gap between the real and the binary and Boolean rank is considerably smaller,
and the result in Theorem~\ref{theodoublegap} is tight.

\begin{theorem}
\label{theoremLowerBound}
Let $M$ be a $0,1$ matrix with $\Rreal(M) = d$. For $d = 3,4$ it holds that $d \leq \Rbin(M) \leq 2d-2$ and $d-1 \leq \Rbool(M), i(M) \leq 2d-2$.
\end{theorem}

The bounds given in Theorem~\ref{theoremLowerBound} are tight: there exist matrices whose binary and Boolean rank match the lower bounds given,
and others for which the upper bound is the correct value.
The proof of Theorem~\ref{theoremLowerBound} reveals other interesting facts about matrices with real rank $d = 3,4$.
Specifically, we show that there does not exist a basic $0,1$ matrix of size $5 \times 5$ and real rank $3$.
For $d = 4$ we prove that any basic $0,1$ matrix of size $7 \times 7$ and real rank $4$ has binary rank at most $6$,
and there does not exist a basic $0,1$ matrix of size $8 \times 8$ with real rank $4$.

It is interesting to note  that Shitov~\cite{shitov2014upper} proved that if $M$ is a non-negative matrix of size $n \times n$, $n > 6$,
and $\Rreal(M) = 3$, then the non-negative rank of $M$ is at most $6n/7$.
Shitov also showed that there exist non-negative matrices of size $n \times n$ with real rank $3$ and non-negative rank $n$, for $n = 3,4,5,6$.

A useful concept in our proofs and analysis is the {\em kernel} of $M$,
which we define as a sub-matrix of $M$ achieved by deleting all zero rows and columns and all duplicate rows and columns of $M$.
Studying the structure of the kernel of a matrix $M$ of real rank $3$ allows us to prove the following.

\begin{theorem}
\label{Rank3Structure}
Let $M$ be a $0,1$ matrix with $\Rreal(M) = 3$. Then:
\begin{itemize}
\item
The kernel of $M$ has at most $4$ rows or columns.
\item
$\Rbin(M) = \Rbool(M) = 4$  if and only if the kernel of $M$ includes $C_4$ as a sub-matrix.
\item
$i(M) = \Rbin(M) = \Rbool(M)$  unless the kernel of $M$ is the following sub-matrix, in which case $\Rbin(M) = 3, \Rbool(M) = i(M) = 2$:
$$
 \left(
  \begin{array}{ccc}
    1 & 1 & 0 \\
    0 & 1 & 1 \\
    1 & 1 & 1 \\
  \end{array}
\right)
$$
%In this case $\Rbin(M) = 3, \Rbool(M) = i(M) = 2$.
\end{itemize}
\end{theorem}

Our proofs of Theorems~\ref{theoremLowerBound} and~\ref{Rank3Structure}, as mentioned earlier, give additional interesting insights as to the structure of basic $0,1$ matrices
with real rank $d = 3,4$. 
We characterize the binary and Boolean rank and the size of largest isolation sets of basic $0,1$ matrices of size $6 \times 6$ and real rank $4$
(see Lemma~\ref{characterize6}, Section~\ref{App4}).
Moreover, we prove that the only $0,1$  matrix of size $n \times n$, for $n = 4,6$, with
$\Rreal(M) =   1+ n/2 $,  $\Rbin(M) = \Rbool(M) = i(M) = n$, is the matrix $C_n$ (see Corollary~\ref{C4} and Corollary~\ref{C6}).
The proof of Theorem~\ref{theoremLowerBound} shows that the bounds in Theorem~\ref{theodoublegap} are tight also for $0,1$ matrices of size $7 \times 7$,
and that the kernel of a $0,1$ matrix $M$ with $\Rreal(M) = 4$ has at most $7$ rows or columns.

Our results give tight bounds on $bp(G)$ and $bc(G)$ for a bipartite graph $G$ whose reduced adjacency matrix $M$ has real rank $3$ or $4$.
The kernel of the matrix $M$ corresponds to the kernel of the graph $G$ as used in~\cite{fleischner2009covering} in the context of parameterized algorithms.
There the kernel is a sub-graph of $G$ obtained by removing all vertices with no neighbors  (which correspond to all-zero rows or columns in $M$),
and removing all vertices but one from any set of vertices with the same neighbor set (which corresponds to removing duplicate rows and columns in $M$).
Thus, the number of rows and columns of the kernel of $M$ is the same as the number of vertices in the kernel of $G$.
As stated in~\cite{fleischner2009covering}, the number of vertices on each side of the kernel is at most $2^{bp(G)}$,
thus achieving a kernel with at most $2^{bp(G)+1}$ vertices in total.
As our results show, for a reduced adjacency matrix $M$ with real rank $d = 3,4$, the number of vertices in the kernel of $G$ is significantly smaller.
Translating Theorems~\ref{theoremLowerBound} and~\ref{Rank3Structure} into the language of graphs, and noticing that
the circulant matrix $C_4$ is the reduced adjacency matrix of a cycle of length $8$, we get the following for real rank $3$:

\begin{corollary}
Let $G $ be a bipartite graph whose reduced adjacency matrix $M$ has real rank $3$. Then:
\begin{itemize}
\item
 $3 \leq bp(G) \leq 4$ and $2 \leq bc(G)\leq 4$.
\item
 $bp(G) = bc(G)=4$  if and only if the kernel of $G$ has a cycle of length $8$ as a subgraph.
\item
$bp(G) = bc(G)$ unless the kernel of $G$ is a cycle of length $6$ with a chord bisecting the cycle.
 \item
 The kernel of $G$ has at most $4$ vertices on one of its sides.
\end{itemize}
\end{corollary}

Finally, our results can also give efficient approximation algorithms up to a small additive constant for the binary and Boolean rank when $1 \leq \Rreal(M) \leq 4$.
The approximation algorithm can simply compute the real rank of the kernel of the matrix and output it as an approximation to the binary and Boolean rank,
where for $\Rreal(M) = 3$ it can even determine the binary and Boolean rank exactly by determining the structure of the kernel as described in Theorem~\ref{Rank3Structure}.

We conclude with a few open problems. First, what is the maximal real rank $d$ for which the bound in Theorem~\ref{theoremLowerBound} holds,
that is, the binary rank is at most $2d-2$? We also wonder how far can a computer search assist us here.
We have preliminary results that there exist basic $0,1$ matrices of size $11 \times 11$ with real rank $d = 5$.
However, the binary rank of the matrices we found so far is at most $8 = 2d-2$.
Is $8$ the largest binary rank possible for real rank $5$ as exists in the matrix $C_8$, and is there a $12 \times 12$ basic $0,1$ matrix with real rank $5$?
Another interesting problem is to determine for what  $n$ is $C_n$ the only matrix of size $n \times n$ with
$\Rreal(M) =   1+ n/2 $,  $\Rbin(M) = \Rbool(M) = i(M) = n$.

\subsection{Our techniques}
\label{IntroTechniques}
We prove our results using both algebraic and combinatorial techniques combined with a computer program which assisted us in achieving some of the theoretical results.
Since the number of possible matrices we had to examine grows exponentially with the size of the matrix,
we use the following algorithmic ideas to prune the number of options considered by the computer program:
\begin{enumerate}
\item
It is enough to consider basic $0,1$ matrices. This reduces significantly the number of matrices we had to consider.
\item
A $0,1$ matrix has the {\em $2$-sum property} if it has two rows (columns) whose sum equals a third row (column).
The binary rank of such a matrix cannot be full (see Lemma~\ref{sumoftworows}),
and thus, we can remove all these matrices when looking for a matrix with full binary rank.
\item
Two matrices will be called {\em equivalent} if they are identical up to a permutation of the rows and columns or the transpose operation.
Our algorithm will consider only one matrix from each class of equivalent matrices, and this matrix will be a {\em representative} for this class.
\item
Each row of a matrix was labeled by an integer which corresponds to the binary number represented by this row vector.
This allowed us to sort the rows of each matrix, and then remove efficiently matrices which are identical up to a permutation of the rows. The same can be done for the columns.
\end{enumerate}
Our proof of Theorem~\ref{theoremLowerBound} for matrices with real rank $d = 3$ uses solely algebraic and combinatorial ideas.
The main ingredient is proving that there does not exist a basic $0,1$ matrix of size $5 \times 5$ with real rank $3$.
This can be done by using the fact that each matrix with real rank $3$ must include a $3 \times 3$ sub-matrix of real rank $3$,
and as the number of possible representatives of matrices of size $3 \times 3$ and real rank $3$ is relatively small,
it is possible to go over each such representative and prove that it can not be extended to a $5 \times 5$ basic $0,1$ matrix with real rank $3$.

However, when  $d = 4$ the number of representatives of matrices of size $4 \times 4$ and real rank $4$ becomes significantly larger.
Therefore, we used the algorithmic ideas described above to write a computer program that checks what is the largest $k$
for which there exists a basic $0,1$ matrix of size $k \times k$ with real rank $d$, and for such matrices get an upper bound on their binary rank, using combinatorial and algebraic arguments.
The program starts by considering all basic $0,1$ matrices of size $d \times d$ and real rank $d$,
and then prunes this set of matrices as explained above to get just one representative of each class of equivalent matrices.
Then the program tries to add additional rows and columns to each representative,
so as to get larger matrices while not increasing the real rank.
In each step when a new class of larger matrices is found, the matrices are pruned again so that only the new representatives remain
and carried over to the next step.
For $d = 4 $ we show that there are no basic $0,1$ matrices of size $8 \times 8$ with real rank $4$, and prove that any basic matrix with at most
$7$ rows or columns and real rank $4$ has binary rank at most $6$.

\section{Preliminaries}
\label{prliminaries}

We now state a few simple lemmas  used throughout the paper, and also prove Theorem~\ref{theodoublegap}.
Since identical rows and columns and all-zero rows and columns do not change the rank,
it is enough to consider {\em basic} matrices, and for any matrix which is not basic, consider its {\em kernel} as defined in the introduction.

\iffalse
\begin{definition}
A $0,1$ matrix is called {\em basic} if all rows are distinct and all columns are distinct and there are no all-zero rows or columns.
\end{definition}

\begin{definition}
The {\em kernel} of $M$ is a sub-matrix achieved by deleting all duplicate rows and columns and all zero rows and columns of $M$.
\end{definition}

\begin{lemma}
\label{basic_large_a}
Let $M$ be a basic $0,1$ matrix of size $n \times m$ with $\Rreal(M) = d$.
Then every sub-matrix $A$ of $M$ such that $\Rreal(A) = d$, is basic.
\end{lemma}
\begin{proof}
Recall the following simple fact: If $X$ is an $n \times d$ matrix whose columns are independent, and  $\alpha_1, \alpha_2 \in \mathbb{R}^d$ are two vectors,
then  $X\alpha_1 = 0$ if and only if $\alpha_1 = 0$, and $X\alpha_1 = X\alpha_2$ if and only if $\alpha_1 = \alpha_2$.
From this follows the observation that a $0,1$ matrix $M'$ is basic if and only if for every optimal decomposition $M' = X \cdot Y$,
both $X$ and $Y$ do not contain all-zero or identical rows and columns, respectively.

Next, let $M = X\cdot Y$ be an optimal decomposition of $M$, and let $A$ be a sub-matrix of $M$ with $\Rreal(A) = d$.
Let $X'$ contain the rows of $X$ that correspond to the rows of $A$, and similarly
let $Y'$ contain the corresponding columns of $Y$. Since the rank of $A$ is $d$,
it follows, by definition, that $X'\cdot Y'$ is an optimal decomposition of $A$. Hence, by the above observation, $A$ is basic,
as obviously both $X'$  and $Y'$ do not contain all-zero or identical rows or columns.
\end{proof}
\fi

\begin{lemma}
\label{basic_large_a}
Let $M$ be a basic $0,1$ matrix of size $n \times m$ with $\Rreal(M) = d$.
Then every sub-matrix $A$ of $M$ such that $\Rreal(A) = d$, is basic.
\end{lemma}
\begin{proof}
Recall the following simple fact: If $X$ is an $n \times d$ matrix whose columns are independent, and  $\alpha_1, \alpha_2 \in \mathbb{R}^d$ are two vectors,
then  $X\alpha_1 = 0$ if and only if $\alpha_1 = 0$, and $X\alpha_1 = X\alpha_2$ if and only if $\alpha_1 = \alpha_2$.
From this follows the observation that a $0,1$ matrix $M'$ is basic if and only if for every optimal decomposition $M' = X \cdot Y$,
both $X$ and $Y$ do not contain all-zero or identical rows and columns, respectively.

Next, let $M = X\cdot Y$ be an optimal decomposition of $M$, and let $A$ be a sub-matrix of $M$ with $\Rreal(A) = d$.
Let $X'$ contain the rows of $X$ that correspond to the rows of $A$, and similarly
let $Y'$ contain the corresponding columns of $Y$. Since the rank of $A$ is $d$,
it follows, by definition, that $X'\cdot Y'$ is an optimal decomposition of $A$. Hence, by the above observation, $A$ is basic,
as obviously both $X'$  and $Y'$ do not contain all-zero or identical rows or columns.
\end{proof}

\begin{lemma}
\label{basic_large}
Let $M$ be a basic $0,1$ matrix of size $n \times m$ with $\Rreal(M) = d$.
Then for any $s,t$ such that $d \leq s \leq n, d \leq t \leq m$, there exists a basic sub-matrix $A$ of $M$ of size $s \times t$ such that $\Rreal(A) = d$.
\end{lemma}
\begin{proof}
The lemma follows from Lemma~\ref{basic_large_a} and the fact that $M$ contains a $s \times t$  sub-matrix $A$ of rank $d$,
for every $s,t$ such that $d \leq s \leq n, d \leq t \leq m$.
\end{proof}

\begin{lemma}
\label{zeroequalrow}
Let $M$ be a basic $0,1$ matrix and let $M'$ be a matrix achieved by augmenting $M$ with a row which is identical to one of the rows of $M$ or with an all-zero row,
and then adding a column to the augmented matrix. If $M'$ is basic then $\Rreal(M') > \Rreal(M)$.
\end{lemma}
\begin{proof}
Let $X$ be the row we augmented $M$ with, let $Y$ be the new column added to the augmented matrix, and let $X'$ be the new row of $M'$ which includes $X$ and the value $x$
at the  intersection of $X'$ and $Y$.
Assume first that $X$ is the all-zero row. Then $x = 1$, otherwise, $X'$ is also an all zero row, in contradiction to $M'$ being basic.
We can subtract $X'$ from all rows of $M'$ which have a non-zero value in $Y$. The resulting matrix $M''$ has real rank which is strictly larger than that of $M$.

In a similar way, let $X$ be identical to one of the rows $R$ of $M$, and let $R'$ be the row $R$ extended with the additional value in $Y$ in this row.
Rows $R'$ and $X'$ must be different in the last position, that is, in their intersection with column $Y$, since $M'$ is basic.
Now we can subtract $R'$ from  $X'$ and continue as before, since the last row is now all-zero but the last value.
\end{proof}

\begin{definition}
A matrix $M$ has the {\em $k$-sum property} if it has $k$ rows (columns) whose sum equals another row (column) in $M$.
\end{definition}

\begin{lemma}
\label{sumoftworows}
Let $M$ be a $0,1$ matrix which has the $k$-sum property.
Then it is possible to partition all ones in the $k+1$ rows (columns) involved in this sum into at most  $k$ monochromatic rectangles of ones.
\end{lemma}
\begin{proof}
Let $X$ be a row in $M$ which is the sum of rows $R_1,...,R_k$.
Since $M$ is a $0,1$ matrix, the ones in $R_1,...,R_k$ are in distinct columns, and row $X$ contains exactly the ones in these $k$ rows.
Thus, the ones in $R_1$ and the ones in $X$ that belong to the same columns define one monochromatic rectangle,
the ones in $R_2$ and the ones in $X$ in the same columns define a second rectangle and so on, and we get a total of $k$ rectangles as required.
\end{proof}

\begin{lemma}
\label{sumofrowaugmented}
Let $M$ be a real matrix, and let $M' = (M|X)$ be a matrix achieved by augmenting $M$ with a column $X$.
If $\Rreal(M) = \Rreal(M')$ then any dependence between the rows of $M$ is preserved in $M'$.
A similar claim holds when the role of the columns and rows is reversed.
\end{lemma}
\begin{proof}
A linear dependence between the rows of $M$ can be expressed as $\alpha^tM = 0$, for some vector of coefficients $\alpha$.
To prove the claim it is enough the show that also $\alpha^t X = 0$. Indeed,
if $\Rreal(M) = \Rreal(M')$ then $X = M\beta$ for some other vector $\beta$,
and therefore $\alpha^t X = \alpha^t (M\beta) = (\alpha^t M)\beta = 0$.
\end{proof}

The following lemmas determine the Boolean and binary rank of a matrix with real rank  $d = 1,2$,
and prove that for $d \geq 3$ the Boolean rank can be smaller than the real rank, however, for $d = 3,4$ it is at least $d-1$.

\begin{lemma}
\label{real12}
For any $0,1$ matrix $M$ it holds:
\begin{itemize}
\item
 $\Rreal(M) = 1$ if and only if $\Rbool(M) = 1$ if and only if $\Rbin(M) = 1$.
 \item
$\Rreal(M) = 2$ if and only if  $\Rbin(M) = 2$, and if $\Rreal(M) = 2$ then $\Rbool(M) = 2$.
\end{itemize}
\end{lemma}
\begin{proof}
The first item is easy to verify from the definition. We now prove the second item.
If $\Rbin(M) = 2$ then $\Rreal(M) = 2$ since $\Rreal(M) \leq \Rbin(M)$ and if $\Rreal(M) = 1$ then $\Rbin(M) = 1$ by the first item.

Assume now that $\Rreal(M) = 2$. Thus, $M$ has two distinct non-zero rows $A,B$ such that any other row $X$ is a linear combination of $A$ and $B$.
Let $X = \alpha A + \beta B$. The possible values for $\alpha$ and $\beta$ are $1,-1$ or $0$, as we now show.
Since $A,B$ are distinct, there exists a position in which $A$ and $B$ differ, where
one has a zero in this position and the other has a one. Assume without loss of generality that $B$ has a $1$. Thus, $\beta$ is either $0$ or $1$.
It is easy to verify that $\alpha$ is restricted to $1,-1, 0$.
Therefore, all other rows of $M$ are one of $0,A,B,A + B, B-A$.
Also note, that rows $A + B$ and $B-A$ cannot coincide in $M$.
Thus, we can choose $A$ and $B$ so that all other rows are one of $0,A,B,A + B$.
Hence, by Lemma~\ref{sumoftworows} we can cover and partition $M$ with at most $2$ rectangles.
By the first item of this lemma, we can again rule out the option that $\Rbool(M)$ = 1.
\end{proof}

\begin{lemma}[Beasley~\cite{BEASLEY20123469}]
\label{lemmaBeasley}
For $k = 1,2$, $i(M) = k$ if and only if $\Rbool(M) = k$.
\end{lemma}

\begin{lemma}
\label{RealnBoolnminus1}
For any $d \geq 3$ there exists a basic $0,1$ matrix $M$  with $\Rreal(M) = d$ and $\Rbool(M) = i(M) = d-1$.
\end{lemma}
\begin{proof}
Consider the block matrix $M$ of size $d \times d$ which has two blocks on the main diagonal and zeros elsewhere.
One block is the identity matrix of size $(d-3)\times (d-3)$,
and the other is following  $3 \times 3$ matrix $A$:
$$
A = \left(
  \begin{array}{ccc}
    1 & 1 & 0 \\
    0 & 1 & 1 \\
    1 & 1 & 1 \\
  \end{array}
\right)
$$
It is easy to verify that $\Rreal(A) = 3, \Rbool(A) = i(A) = 2$.
Thus, $\Rreal(M) = (d-3) + 3 = d$ and  $\Rbool(M) = i(M) = (d-3) + 2 = d-1$.
\end{proof}

\begin{lemma}
\label{RealBool34}
If $\Rreal(M) = d$,  $d = 3,4$, then $\Rbool(M), i(M) \geq d-1$.
\end{lemma}
\begin{proof}
Let $d =3$. Assume by contradiction that $\Rbool(M) = 1$. Then by Lemma~\ref{real12}, $\Rreal(M) = 1$.
Thus, $\Rbool(M) \geq 2$, and by Lemma~\ref{lemmaBeasley}, $i(M) \geq 2$.
%a result of Beasley~\cite{BEASLEY20123469} who proved that $i(M) = 2$ if and only if $\Rbool(M) = 2$.

Now let $d  = 4$ and assume by contradiction that $\Rbool(M) \leq 2$. Therefore, the ones of $M$ can be covered by at most $2$ monochromatic rectangles, and
so $M$ has at most $3$ different rows that are not all zeros (two rows, each including exactly one of the $2$ rectangles, and a row including both rectangles).
But then the real rank of $M$ is at most $3$ and we get a contradiction. Thus, $\Rbool(M) \geq 3$, and using Lemma~\ref{lemmaBeasley} we get that $i(M) \geq 3$ as well.
\end{proof}

\definecolor{bluegray}{rgb}{0.4, 0.6, 0.8}

\begin{figure}[htb!]
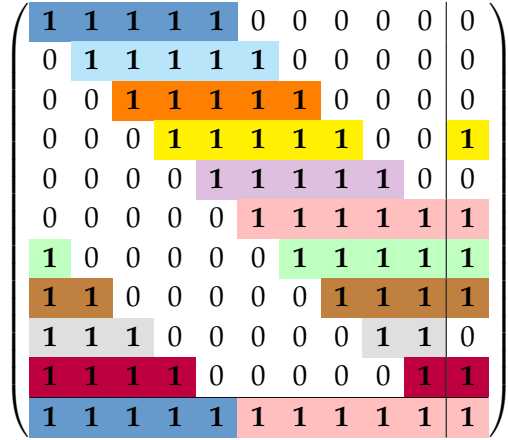

%\captionsetup{width=0.9\textwidth}
\centering
%\begin{tabular}{cc}
$
\left(
  \begin{array}{cccccccccc|c}
    \cellcolor{bluegray}{\bf 1} & \cellcolor{bluegray}{\bf 1} & \cellcolor{bluegray}{\bf 1} & \cellcolor{bluegray}{\bf 1} & \cellcolor{bluegray}{\bf 1} & 0 & 0 & 0 & 0 & 0 & 0 \\
    0 & \cellcolor{cyan!25}{\bf 1} & \cellcolor{cyan!25}{\bf 1} &\cellcolor{cyan!25}{\bf 1} & \cellcolor{cyan!25}{\bf 1} & \cellcolor{cyan!25}{\bf 1} & 0 & 0 & 0 & 0 & 0 \\
    0 & 0 & \cellcolor{orange}{\bf 1} & \cellcolor{orange}{\bf 1} & \cellcolor{orange}{\bf 1} & \cellcolor{orange}{\bf 1} & \cellcolor{orange}{\bf 1} & 0 & 0 & 0 & 0\\
    0 & 0 & 0 & \cellcolor{yellow}{\bf 1} & \cellcolor{yellow}{\bf 1} & \cellcolor{yellow}{\bf 1} & \cellcolor{yellow}{\bf 1} & \cellcolor{yellow}{\bf 1} & 0 & 0 & \cellcolor{yellow}{\bf 1} \\
    0 & 0 & 0 & 0 & \cellcolor{violet!25}{\bf 1} & \cellcolor{violet!25}{\bf 1} & \cellcolor{violet!25}{\bf 1} & \cellcolor{violet!25}{\bf 1} & \cellcolor{violet!25}{\bf 1} & 0 & 0 \\
    0 & 0 & 0 & 0 & 0 & \cellcolor{red!25}{\bf 1} & \cellcolor{red!25}{\bf 1} & \cellcolor{red!25}{\bf 1} & \cellcolor{red!25}{\bf 1} & \cellcolor{red!25}{\bf 1} & \cellcolor{red!25}{\bf 1} \\
    \cellcolor{green!25}{\bf 1} & 0 & 0 & 0 & 0 & 0 & \cellcolor{green!25}{\bf 1} & \cellcolor{green!25}{\bf 1} & \cellcolor{green!25}{\bf 1} & \cellcolor{green!25}{\bf 1} & \cellcolor{green!25}{\bf 1}  \\
    \cellcolor{brown}{\bf 1} & \cellcolor{brown}{\bf 1} & 0 & 0 & 0 & 0 & 0 & \cellcolor{brown}{\bf 1} & \cellcolor{brown}{\bf 1} & \cellcolor{brown}{\bf 1} & \cellcolor{brown}{\bf 1} \\
    \cellcolor{gray!25}{\bf 1} & \cellcolor{gray!25}{\bf 1}  & \cellcolor{gray!25}{\bf 1} & 0 & 0 & 0 & 0 & 0 & \cellcolor{gray!25}{\bf 1} & \cellcolor{gray!25}{\bf 1} & 0 \\
    \cellcolor{purple}{\bf 1} & \cellcolor{purple}{\bf 1}  & \cellcolor{purple}{\bf 1} & \cellcolor{purple}{\bf 1} & 0 & 0 & 0 & 0 & 0 & \cellcolor{purple}{\bf 1}  & \cellcolor{purple}{\bf 1} \\ \hline
    \cellcolor{bluegray}{\bf 1} & \cellcolor{bluegray}{\bf 1} & \cellcolor{bluegray}{\bf 1} & \cellcolor{bluegray}{\bf 1} & \cellcolor{bluegray}{\bf 1} &   \cellcolor{red!25}{\bf 1} &   \cellcolor{red!25}{\bf 1} & \cellcolor{red!25}{\bf 1} & \cellcolor{red!25}{\bf 1} & \cellcolor{red!25}{\bf 1} & \cellcolor{red!25}{\bf 1}\\
  \end{array}
\right)
$
 % \end{tabular}
\caption{A matrix of size $(2k+1)\times (2k+1)$ built by adding a row and a column to the circulant matrix $C_{2k}$, where $k = 5$.
A partition of the ones into $2k$ rectangles is described, where each rectangle is represented by a different color,
and since the first $2k$ ones on the main diagonal are an isolation set of size $2k$,  the binary rank is exactly $2k$.}
\label{circulant}	
\end{figure}

The following lemma was proved by de Caen,  Gregory, Henson,  Lundgren, and Maybee~\cite{realnonnegative}, and deals with even sized matrices.
As we prove in Lemma~\ref{Realkplus1Bin2k1} their result can be generalized to odd sized matrices.
Theorem~\ref{theodoublegap} follows.

\begin{lemma}
\label{Realkplus1Bin2k}
For $k \geq 2$, let $C_{2k}$ be the $2k \times 2k$ basic circulant $0,1$ matrix defined by a row with $k$ consecutive ones followed by $k$ zeros.
Then $\Rreal(C_{2k}) = k+1$,  $\Rbin(C_{2k}) = \Rbool(C_{2k}) = i(C_{2k}) = 2k$.
\end{lemma}
\iffalse
\begin{proof}
For the proof that $\Rreal(C_{2k}) = k+1$ see~\cite{realnonnegative} or~\cite{haviv2023binary}.
As to the binary and Boolean rank, the $2k$ ones on the main diagonal of $C_{2k}$ are a maximal isolation set of size $2k$,
and since the ones in each row can be covered by a single monochromatic rectangle of ones then $\Rbin(C_{2k}) = \Rbool(C_{2k})  = 2k$.
\end{proof}
\fi

\begin{lemma}
\label{Realkplus1Bin2k1}
For $k \geq 3$, there exists a basic $0,1$ matrix $M$ of size $(2k + 1)\times (2k+ 1)$, such that $\Rreal(M) = k+1$,  $\Rbin(M) = \Rbool(M) = i(M) = 2k$.
\end{lemma}
\begin{proof}
Let $C_{2k}$ be the circulant matrix  described in Lemma~\ref{Realkplus1Bin2k}.
We show how to add a row and a column to $C_{2k}$ and get a matrix $M$ of size  $(2k+1) \times (2k+1)$ with the claimed properties.
%Denote by $X_i$ and $Y_i$ the rows and columns of the matrix $M$.

First add to $C_{2k}$ a row which is the sum of rows $1$ and $k+1$ of $C_{2k}$. The resulting new row is the all-one vector.
Next add to the augmented matrix a new column $Y_{2k+1} = Y_{2k} - Y_{2k-1} + Y_{2k-2}$, which is a linear combination of the last three columns $Y_{2k-2},Y_{2k-1},Y_{2k}$
of the augmented matrix.
The new column is $(0,0,...,0,1,0,1,1,...,1,0,1,1)^T$, where the first block of zeros and the middle block of ones are each of length $k-2$.
Column $Y_{2k+1}$ is well defined since $k - 2 \geq 1$ as $k \geq 3$, and it is easy to verify from the structure of the circulant matrix $C_{2k}$ that $Y_{2k+1}$ contains only zeros and ones
(see Figure~\ref{circulant}).

The matrix $M$ is basic since $C_{2k}$ is basic and the new row and column are different from all rows and columns of $C_{2k}$ (again since $k - 2 \geq 1$).
Furthermore, $\Rreal(M) = \Rreal(C_{2k}) = k+1$, since the new row and column are a linear combination of rows and columns of $C_{2k}$.
Finally, $\Rbin(M) = \Rbool(M) = 2k$  since the first $2k$ ones on the main diagonal of $M$ are an isolation set of size $2k$, and using
Lemma~\ref{sumoftworows} it is possible to partition the ones of $M$ into $2k$ monochromatic rectangles since the last row of $M$ is also a sum of two rows
(see Lemma~\ref{sumofrowaugmented}). That is, the ones on rows $1,k+1,2k+1$ can be partitioned into $2$ rectangles,
and the ones of the remaining $2k-2$ rows can be partitioned into $2k-2$ additional rectangles (see Figure~\ref{circulant}).
\end{proof}

\section{Real Rank 3}
\label{Sec3}

In this section we consider $0,1$ matrices with real rank $3$.
The main observation is that any $0,1$ matrix $M$ with $\Rreal(M) = 3$  must contain a $3 \times 3$ sub-matrix of real rank $3$ which is equivalent to
one of the $7$ representatives  in Figure~\ref{fig:rank3}.
This allows us to partition the basic matrices of real rank $3$ into $7$ classes, analyze each class separately by considering only its representative,
and prove that it is not possible to augment any of these representatives into a $5 \times 5$  basic $0,1$ matrix with real rank $3$ (Lemma~\ref{lemma3}).
From this will follow Theorem~\ref{theoremLowerBound} for $d = 3$. These representatives also help us characterize the maximal size of
an isolation set according to the structure of the kernel of $M$, thus proving Theorem~\ref{Rank3Structure}.

\begin{figure}[htb!]
%\captionsetup{width=0.9\textwidth}
\begin{center}
\begin{tabular}{cccc}
   % after \\: \hline or \cline{col1-col2} \cline{col3-col4} ...
$ A_1 = \left(
  \begin{array}{ccc}
    1 & 0 & 0 \\
    0 & 1 & 0 \\
    0 & 0 & 1 \\
  \end{array}
\right) $ &
$ A_2 = \left(
  \begin{array}{ccc}
    1 & 0 & 0 \\
    0 & 1 & 0 \\
    0 & 1 & 1 \\
  \end{array}
\right) $
 &
 $ A_3 = \left(
  \begin{array}{ccc}
    1 & 0 & 0 \\
    0 & 1 & 0 \\
    1 & 1 & 1 \\
  \end{array}
\right) $
&
$A_4 = \left(
  \begin{array}{ccc}
    1 & 0 & 0 \\
    1 & 1 & 0 \\
    0 & 1 & 1 \\
  \end{array}
\right) $
\\ \\  \\
 $A_5 = \left(
  \begin{array}{ccc}
    1 & 0 & 0 \\
    1 & 1 & 0 \\
    1 & 1 & 1 \\
  \end{array}
\right) $
&
$ A_6 =  \left(
  \begin{array}{ccc}
    1 & 1 & 0 \\
    0 & 1 & 1 \\
    1 & 0 & 1 \\
  \end{array}
\right) $
&
$A_7 = \left(
  \begin{array}{ccc}
    1 & 1 & 0 \\
    0 & 1 & 1 \\
    1 & 1 & 1 \\
  \end{array}
\right) $
&
  \end{tabular}
\end{center}
\caption{The seven possible representatives of matrices of real rank $3$ and size $3 \times 3$.}
\label{fig:rank3}	
\end{figure}

The following lemma can be verified by a computer program, but it can also be proved using only algebraic and combinatorial arguments.
We include the full theoretical proof since the analysis is required for the proof of Theorem~\ref{Rank3Structure}.

\begin{lemma}
\label{lemma3}
 There does not exist a $5 \times 5$ basic $0,1$ matrix $M$ with  $\Rreal(M) = 3.$
\end{lemma}
\begin{proof}
Assume by contradiction that there exists such a matrix $M$.
Then $M$ has a sub-matrix of size $3 \times 3$ and real rank $3$, which is equivalent to one of the seven sub-matrices $A_i$ in Figure~\ref{fig:rank3}.
It is easy to verify that the matrices $A_i$, $1 \leq i \leq 7$, are the only $7$ possible representatives of size $3 \times 3$ and real rank $3$,
by considering all options of the possible number of ones in each row and column and making sure that the three rows and columns are independent.

We now show that it is not possible to extend any of these seven representatives into a basic $0,1$ matrix $M$ of size $5 \times 5$ with $\Rreal(M) = 3$.
We first augment each matrix $A_i$ with two additional rows  which are a linear combination of the existing rows of $A_i$,
and then show that it is not possible to add two additional distinct non-zero columns to the augmented sub-matrix without increasing the real rank.
By Lemma~\ref{zeroequalrow} if we augment $A_i$ with a row which is identical to one of the rows of $A_i$ or augment $A_i$ with an all zero row,
and then add columns to this augmented matrix, the resulting matrix, if basic, will have a strictly larger real rank than $A_i$.
Thus, any row added to $A_i$ should be a non-trivial linear combination of at least two rows of $A_i$.
Since the rows of $A_i$ are a base, they span all $0,1$ row vectors, but we explicitly find  the possible non-trivial linear combinations of the rows of $A_i$,
since the coefficients in each combination determine the columns that can be added to the augmented matrix (see Lemma~\ref{sumofrowaugmented}).

{\bf Case 1:}
The only possible linear combinations  of two or more rows of $A_1$ are a simple sum and give the following row vectors: $(1,1,0), (0,1,1), (1,0,1), (1,1,1)$.
Consider first the matrix $A'_1$ achieved by extending  $A_1$ with the two rows: $(1,1,0), (0,1,1)$,
and  assume by contradiction that $A'_1$ can be augmented with two additional columns $X_1,X_2$, so that the resulting $5 \times 5$ matrix is basic with real rank $3$.

First note, that for $i = 1,2$, there is at least one $1$ in the first three positions of $X_i$.
Otherwise, $X_i$ is the all zero vector as the last two positions will also be a zero by Lemma~\ref{sumofrowaugmented}, and thus the resulting matrix is not basic.
Furthermore, $X_i$ has at least $2$ ones in the first three positions, otherwise, these three positions are identical to one
of the columns of $A_1$, and then $X_i$ is identical to one of the columns of $A'_1$ (again using Lemma~\ref{sumofrowaugmented}).
Next, recall that the non-zero coefficients of the linear combinations are $1$, that is
$(1,1,0) = 1 \cdot(1,0,0) + 1\cdot(0,1,0)$ and $(0,1,1) = 1\cdot (0,1,0) + 1 \cdot (0,0,1)$.
Thus, if for example the first row of $A_1$ is extended with a pair of elements $x,y$,
such that $x = 1$ or $y = 1$, then the extension of the second row of $A_1$ cannot contain a $1$ in the same position, and similarly, for the second and third row of $A_1$.
But, this means that we can only extend $A'_1$ with one column, the column $(1,0,1,1,1)^T$, and we get a contradiction.

A similar contradiction is achieved if we extend $A_1$ with the remaining possibilities of two rows, each with two ones.
Finally, consider the case that $A_1$ is extended with a row with two ones $(1,1,0)$ and a row with three ones.
Again, the coefficients of the linear combinations are $1$, and here since row $(1,1,1)$ is a linear combination of all three rows of $A_1$,
then it is possible to have only one $1$ in the first three positions of any new column added.
But then there are identical columns in the augmented matrix and we get a contradiction as required.

{\bf Case 2:}
The possible linear combinations of two or more rows of $A_2$ are:
\begin{eqnarray*}
% \nonumber to remove numbering (before each equation)
  (0,0,1) &=& -1\cdot (0,1,0) + 1 \cdot (0,1,1), \\
  (1,1,0) &=& 1\cdot (1,0,0) + 1 \cdot (0,1,0), \\
  (1,0,1) &=&  1\cdot (1,0,0) -1\cdot (0,1,0) + 1 \cdot (0,1,1),\\
  (1,1,1) &=& 1\cdot (1,0,0) + 1 \cdot (0,1,1).
\end{eqnarray*}
If we extend $A_2$ with row $(0,0,1)$ we get a matrix that contains $A_1$ as a sub-matrix and we can continue as in Case 1.
If we extend $A_2$ with both rows $(1,1,0)$ and  $(1,0,1)$, then again, we cannot add a column which is identical in the first three positions to one of the columns of $A_2$.
We also cannot add a column that has two ones in the first two positions, since row $(1,1,0)$ is the sum of the first two rows of $A_2$.
Similarly, we cannot add a column whose first three positions are $(1,0,1)^T$, since the coefficients in the linear combination which results in row $(1,0,1)$ are $1,-1,1$,
and thus we will get the value $2$.
Finally, we cannot add a column whose first three positions are $(0,1,0)^T$ since then the matrix $M$ will contain the matrix $A_1$ and we can continue as in Case 1.
Thus, there is no way to complete $A_2$, when extended with rows $(1,1,0),(1,0,1)$, into a basic $5 \times 5$ matrix with rank $3$ as required.

Next we try to extend $A_2$ with rows $(1,1,0), (1,1,1)$.
A similar argument shows that there is no way to add columns to the augmented matrix in this case without getting two identical columns or values that are not $0,1$.
Extending $A_2$ with rows $(1,0,1), (1,1,1)$ results in a similar contradiction.

{\bf Case 3:}
The  possible linear combinations of two or more rows of $A_3$ are:
\begin{eqnarray*}
% \nonumber to remove numbering (before each equation)
  (1,1,0) &=& 1\cdot (1,0,0) + 1 \cdot (0,1,0), \\
  (0,1,1) &=& -1\cdot (1,0,0) + 1 \cdot (1,1,1), \\
  (1,0,1) &=& -1\cdot (0,1,0) + 1 \cdot (1,1,1) \\
  (0,0,1) &=& -1\cdot (1,0,0) - 1 \cdot (0,1,0) + 1 \cdot(1,1,1).
\end{eqnarray*}
Extending $A_3$ with $(0,0,1)$ gives us $A_1$ and we are done. Extending $A_3$ with $(1,0,1)$ or with $(0,1,1)$ gives us $A_2$ up to a permutation of rows and columns, and we are done.
Thus, we can only extend $A_3$ with the row $(1,1,0)$. But then we do not get a $5 \times 5$ matrix and again we are done.

{\bf Case 4:}
The  possible linear combinations of two or more rows of $A_4$ are:
\begin{eqnarray*}
% \nonumber to remove numbering (before each equation)
 (0,1,0) &=& -1\cdot (1,0,0) + 1 \cdot (1,1,0), \\
  (1,1,1) &=& 1\cdot (1,0,0) + 1 \cdot (0,1,1), \\
  (0,0,1) &=& 1\cdot (1,0,0) - 1 \cdot (1,1,0) + 1 \cdot(0,1,1),\\
  (1,0,1) & = & 2 \cdot (1,0,0) - 1 \cdot (1,1,0) + 1 \cdot (0,1,1).
\end{eqnarray*}
We cannot extend $A_4$ with $(0,1,0)$ or with $(0,0,1)$ since then we get $A_2$ as a sub-matrix.
If we extend $A_4$ with $(1,0,1),(1,1,1)$, then, as in the previous cases we cannot add two distinct non-zero columns.

{\bf Case 5:}
The  possible linear combinations of two or more rows of $A_5$ are:
\begin{eqnarray*}
% \nonumber to remove numbering (before each equation)
  (0,1,0) &=& -1\cdot (1,0,0) + 1 \cdot (1,1,0),  \\
  (0,1,1) &=& -1\cdot (1,0,0) + 1 \cdot (1,1,1), \\
  (0,0,1) &=& -1\cdot (1,1,0) + 1 \cdot (1,1,1), \\
  (1,0,1) &=& 1\cdot (1,0,0) - 1 \cdot (1,1,0) + 1 \cdot(1,1,1).
\end{eqnarray*}
Extending $A_5$ with $(0,1,0)$ gives $A_3$ as a sub-matrix; extending $A_5$ with $(0,0,1)$ gives $A_2$ as a sub-matrix, and extending $A_5$ with $(0,1,1)$ gives $A_4$ as a sub-matrix.
Thus, we can only extend $A_5$ with $(1,0,1)$ and again we are done.

{\bf Case 6:}
The possible linear combinations of two or more rows of  $A_6$ are:
\begin{eqnarray*}
% \nonumber to remove numbering (before each equation)
  (1,1,1) &=& \frac{1}{2}\cdot (1,1,0) + \frac{1}{2} \cdot (0,1,1) + \frac{1}{2}\cdot (1,0,1), \\
  (1,0,0) &=& \frac{1}{2}\cdot (1,1,0)  -\frac{1}{2} \cdot (0,1,1) + \frac{1}{2}\cdot (1,0,1), \\
  (0,0,1) &=& -\frac{1}{2}\cdot (1,1,0) + \frac{1}{2} \cdot (0,1,1) + \frac{1}{2}\cdot (1,0,1), \\
  (0,1,0) &=& \frac{1}{2}\cdot (1,1,0) + \frac{1}{2} \cdot (0,1,1) - \frac{1}{2}\cdot (1,0,1).
\end{eqnarray*}
If we extend $A_6$ with one of $(1,0,0), (0,1,0), (0,0,1)$ we get a matrix that contains $A_4$ we are done. Thus, we can only extend $A_6$ with $(1,1,1)$ and again we are done.

{\bf Case 7:}
The possible linear combinations of two or more rows of $A_7$ are:
\begin{eqnarray*}
% \nonumber to remove numbering (before each equation)
  (0,0,1) &=& -1\cdot (1,1,0)  + 1 \cdot (1,1,1), \\
  (1,0,0) &=&  -1  \cdot (0,1,1) + 1 \cdot (1,1,1),\\
  (0,1,0) & =& 1\cdot (1,1,0)  + 1 \cdot (0,1,1) - 1 \cdot (1,1,1),\\
  (1,0,1) &=& -1\cdot (1,1,0)  -1 \cdot (0,1,1) + 2 \cdot (1,1,1).
\end{eqnarray*}
Extending $A_7$ with $(0,0,1)$ or with $(1,0,0)$ gives us a matrix that contains $A_4$ as a sub-matrix and we can continue as in case 4. Extending $A_7$ with $(0,1,0)$
gives us a matrix that contains $A_5$ as a sub-matrix and we can continue as in case 5, and thus we can extend $A_7$ only with $(1,0,1)$  and we are done.
\end{proof}

%%%%%%%%%%%%%%%%%%%%%%%%%%%%%%%%%%%%%%

We can now prove Theorem~\ref{Rank3Structure} which states that for a $0,1$ matrix $M$ with $\Rreal(M) = 3$, it holds that $\Rbin(M) = \Rbool(M) = i(M)$,
unless the kernel of $M$ is $A_7$ (in this case, $\Rbin(M) = 3, \Rbool(M) = i(M) = 2$).
Furthermore,  $\Rbin(M) = \Rbool(M) = i(M) = 4$ if and only if the kernel of $M$ includes the $4 \times 4$ circulant matrix $C_4$.

\begin{proof}[of Theorem~\ref{Rank3Structure}]
Let $M'$ be the kernel of $M$. Since $\Rreal(M') = 3$ then $M'$ includes a sub-matrix of size $3 \times 3$ which is equivalent to
one of the seven representatives $A_i$ in Figure~\ref{fig:rank3}.

Assume first that $M'$ has at most $3$ rows or at most $3$ columns. Therefore, $\Rbin(M') \leq 3$.
In all cases, but for the matrix $A_7$, the ones on the diagonal of $A_i$ are an isolation set of size $3$,
and thus, for these cases, the binary and Boolean rank of $M'$ are $3$.
If $M' = A_7$, then  $\Rbin(M') = 3$, but  $\Rbool(M') = i(M') = 2$.
Assume now that $M'$ contains $A_7$ as a sub-matrix, but has at least one more distinct row (a similar argument holds if $M'$ has at least one more column).
We show that in this case $M'$ has an isolation set of size $3$ as claimed. The additional row must be a linear combination of the rows of $A_7$.
As in the proof of Lemma~\ref{lemma3}, the possible linear combinations of rows of $A_7$ are: $(0,0,1), (1,0,0), (0,1,0),(1,0,1)$.
Adding any one of the first two rows to $A_7$ gives us $A_4$ as a sub-matrix, adding $(0,1,0)$ gives us $A_5$ as a sub-matrix,
and adding $(1,0,1)$ gives us $A_6$ as a sum-matrix. In all cases $M'$ contains an isolation set of size $3$ as claimed, and thus its binary and Boolean rank are $3$.

Now assume that $M'$ has at least $4$ rows and at least $4$ columns.
By Lemma~\ref{lemma3} there does not exist a $5 \times 5$ basic $0,1$ matrix with real rank $3$.
Therefore, we can assume that $M'$ has $4$ distinct rows and at least $4$ columns (the case of $4$ columns and at least $4$ rows is treated similarly).
Again $M'$ includes one of the matrices $A_i$ as a sub-matrix.
We consider for each matrix $A_i$ the possible ways to add a row and columns to $A_i$, which are linear combinations of rows and columns of $A_i$,
as described in the proof of Lemma~\ref{lemma3}.
Recall that if the additional row in $M'$  is the sum of two or three rows in $A_i$ then its binary rank is $3$ by Lemma~\ref{sumoftworows},
and since $A_i$ has an isolation set of size $3$ for $1 \leq i \leq 6$, we are done in this case.
Thus, it is enough to consider for each $A_i$, $1 \leq i \leq 6$, the case of adding a row which is not the sum of rows of $A_i$.
For $A_7$ we have to consider also the option of adding a row which is the sum of rows in $A_7$.

{\bf Case 1:} The only non-trivial linear combinations of rows of $A_1$ are sums of two or three rows of $A_1$, and we are done.

{\bf Case 2:} The only non-trivial linear combination of rows of $A_2$ which does not give a row which is a sum of rows is:
 $$(1,0,1) =  1\cdot (1,0,0) -1\cdot (0,1,0) + 1 \cdot (0,1,1).$$
 There are two ways to add a distinct column which corresponds to this linear combination:

$$ M_1 = \left(
  \begin{array}{ccc|c}
    1 & 0 & 0 & 1\\
    0 & 1 & 0 & 1\\
    0 & 1 & 1 & 0\\ \hline
    1 & 0 & 1 & 0 \\
  \end{array}
\right)   \;\;\;\;\;\;\;\;
M_2 = \left(
  \begin{array}{ccc|c}
    1 & 0 & 0 & 1\\
    0 & 1 & 0 & 1\\
    0 & 1 & 1 & 1\\\hline
    1 & 0 & 1 & 1 \\
  \end{array}
\right) $$

The matrix $M_1$ is just a permutation of the rows and columns of the circulant matrix $C_4$. It has binary rank $4$ and an isolation set of size $4$.
The matrix $M_2$ has binary rank $3$ since the last column we added is a sum of the first two columns, and it has an isolation set of size $3$ as claimed.
Of course, it is possible to add both columns, but then we get an augmented matrix of size $4 \times 5$ which includes $M_1$ as a sub-matrix, and the same argument holds.

{\bf Case 3:} The only non-trivial linear combination of the rows of $A_3$ which does not give a row which is a sum of rows is:
$$(0,0,1) = -1\cdot (1,0,0) - 1 \cdot (0,1,0) + 1 \cdot(1,1,1).$$
There is no way to add a column to $A_3$ when augmented with the row $(0,0,1)$ with the coefficients of this linear combination.

{\bf Case 4:} There are two non-trivial linear combinations of rows of $A_4$ which do not give a row which is a sum of rows.
If we extend $A_4$ with the row:
 $$(0,0,1) = 1\cdot (1,0,0) - 1 \cdot (1,1,0) + 1 \cdot(0,1,1),$$
 then the resulting matrix is as in Case 2 up to a permutation of the rows and columns.
 If we extend $A_4$ with:
 $$  (1,0,1)  =  2 \cdot (1,0,0) - 1 \cdot (1,1,0) + 1 \cdot (0,1,1),$$
 then there is no way to add a distinct non-zero column which is consistent with the coefficients of this combination.

{\bf Case 5:} The only non-trivial linear combination of rows of $A_5$ which does not give a row which is a sum of rows  is:
$$ (1,0,1) = 1\cdot (1,0,0) - 1 \cdot (1,1,0) + 1 \cdot(1,1,1).$$
 There are two ways to add a column which corresponds to this linear combination:

$$ M_1 = \left(
  \begin{array}{ccc|c}
    1 & 0 & 0 & 1\\
    1 & 1 & 0 & 1\\
    1 & 1 & 1 & 0\\ \hline
    1 & 0 & 1 & 0 \\
  \end{array}
\right)  \;\;\;\;\;\;\;\;
M_2 = \left(
  \begin{array}{ccc|c}
    1 & 0 & 0 & 1\\
    1 & 1 & 0 & 0\\
    1 & 1 & 1 & 0\\\hline
    1 & 0 & 1 & 1 \\
  \end{array}
\right) $$

In both cases the resulting matrix has binary rank $3$ since there exist two columns whose sum equals a third column, and in both cases there is an isolation set of size $3$.
If we add both columns to $A_5$ we get a matrix which includes $C_4$ as a sub-matrix.

{\bf Case 6:} In this case there are four ways to add a row which is a linear combination of the rows of $A_6$ as specified in Lemma~\ref{lemma3}.
However, there is no way to add a distinct non-zero column which is consistent with the coefficients of these combinations.

{\bf Case 7:} The matrix $A_7$ is the only matrix of the seven matrices $A_i$ which has a maximal isolation set of size $2$ and not $3$.
Therefore, in this case we consider all possible linear combinations of rows of $A_7$, including combinations which give us a row that is a sum of other rows,
such as the following two combinations:
\begin{eqnarray*}
% \nonumber to remove numbering (before each equation)
  (0,0,1) &=& -1\cdot (1,1,0)  + 1 \cdot (1,1,1), \\
  (1,0,0) &=&  -1  \cdot (0,1,1) + 1 \cdot (1,1,1).
\end{eqnarray*}
If we add the row $(0,0,1)$ to $A_7$ there is one possibility to add a column which is consistent with the coefficients of this linear combination:

$$  \left(
  \begin{array}{ccc|c}
    {\bf 1} & 1 & 0 & 0\\
    0 &  {\bf 1} & 1 & 0\\
    1 & 1 & 1 & 1\\\hline
    0 & 0 & 1 &  {\bf 1} \\
  \end{array}
\right) $$
The binary rank of this matrix is $3$ and the $1$s in bold are an isolation set of size $3$ as claimed. A similar argument holds when we augment $A_7$ with the row $(1,0,0)$.

If we augment $A_7$ with the row $(1,0,1) =  -1\cdot (1,1,0)  -1 \cdot (0,1,1) + 2 \cdot (1,1,1)$, which is a linear combination of all three rows of $A_7$,
it is easy to verify that it is not possible to add a new column to the resulting matrix which is consistent with the coefficients of this combination.
If we augment $A_7$ with the row $ (0,1,0)  = 1\cdot (1,1,0)  + 1 \cdot (0,1,1) - 1 \cdot (1,1,1)$ then we can extend the augmented matrix with either or both of the columns
$(0,1,0,1)^T, (1,0,0,1)^T$. If we extend it with only one of these two columns then it is easy to verify that the resulting matrix has an isolation set, binary and Boolean rank of size $3$.
If we extend the matrix with both columns, we get the following matrix which has an isolation set of size $4$ in bold and contains $C_4$ as a sub-matrix.
$$  \left(
  \begin{array}{ccc|cc}
    {\bf 1} & 1 & 0 & 0 & 1 \\
    0 &  1 & 1 & {\bf 1} & 0  \\
    1 & 1 & {\bf 1} & 0 & 0 \\\hline
    0 & 1 & 0 &  1 & {\bf 1}  \\
  \end{array}
\right) $$
The theorem follows.
\end{proof}

%%%%%%%%%%%%%%%%%%%%%%%%%%%

\begin{theorem}
\label{theo3}
For any $0,1$ matrix $M$ with $\Rreal(M) = 3$ it holds that $2 \leq \Rbool(M), i(M) \leq 4$ and $3 \leq \Rbin(M) \leq 4$.
\end{theorem}
\begin{proof}
It is enough to consider basic matrices.
By Lemma~\ref{basic_large} if there exists a basic $0,1$ matrix $M$ of real rank $3$ and size $n \times m$, where $n,m \geq 5$,
then there exists a basic sub-matrix of size $5 \times 5$ and real rank $3$, in contradiction to Lemma~\ref{lemma3}.
Therefore, $M$ has at most $4$ rows or at most $4$ columns, and thus $\Rbool(M), \Rbin(M) \leq 4$.
The theorem follows by combining this upper bound with Lemma~\ref{RealBool34}
and the fact that $\Rreal(M) \leq \Rbin(M)$.
\end{proof}

\begin{corollary}
\label{C4}
The only $0,1$ matrix $M$ of size $4 \times 4$ with $\Rreal(M)  = 3$ and $\Rbin(M) = \Rbool(M) = i(M) = 4$ is the circulant matrix $C_4$.
\end{corollary}
\begin{proof}
If $M$ is basic, the corollary follows from Theorem~\ref{Rank3Structure}.
Otherwise, $M$ has an all-zero row or column or at least two identical rows or columns and thus, $\Rbin(M) \leq 3$.
\end{proof}

\section{Real Rank 4}
\label{App4}

In this section we prove Theorem~\ref{theoremLowerBound} for $d = 4$.
We first prove that any basic $0,1$ matrix $M$ of size $6 \times 6$ with real rank $4$ has a certain constraint on its rows and columns.
Specifically, in each such matrix, either $M$ has the $2$-sum property in which case the binary and Boolean rank are at most $5$,
or there exist two pairs of rows or two pairs of columns, such that the sum of each pair is the all one vector.
Furthermore, we show that any basic matrix of size $6 \times 6$ that does not have the $2$-sum property is equivalent to one of $4$ representatives,
and we also find the size of the maximal isolation sets of these representatives (see Lemma~\ref{characterize6}).
Then we prove that any basic $7 \times 7$ matrix obtained from these $4$ representatives by adding a row and column
without increasing the real rank, has the $2$-sum property.
From this we can deduce that any basic $7 \times 7$ matrix of real rank $4$ has binary rank at most $6$.
Finally, we show that there does not exist a basic $0,1$ matrix of size $8 \times 8$ and real rank $4$,
and can conclude that the binary and Boolean rank of any $0,1$ matrix with real rank $4$ is at most $6$.

\begin{figure}[htb!]
%\captionsetup{width=0.9\textwidth}
$$ A_1 =
 \left(
 \begin{array}{cccccc}
\cellcolor{blue!25}{\bf 1}& \cellcolor{blue!25}{\bf 1}& 0& 0 &\cellcolor{blue!25}{\bf 1}&    0 \\
0 & \cellcolor{yellow!25}{\bf 1}& \cellcolor{gray!25}{\bf 1} & 0 &\cellcolor{gray!25}{\bf 1}&  \cellcolor{yellow!25}{\bf 1} \\
0 & 0 & \cellcolor{purple!25}{\bf 1} & \cellcolor{purple!25}{\bf 1} &0 & \cellcolor{purple!25}{\bf 1} \\
\cellcolor{orange!25}{\bf 1} & \cellcolor{yellow!25}{\bf 1}& 0& \cellcolor{orange!25}{\bf 1} & 0 &  \cellcolor{yellow!25}{\bf 1} \\
0 & 0& \cellcolor{gray!25}{\bf 1} & 0& \cellcolor{gray!25}{\bf 1}& 0 \\
\cellcolor{orange!25}{\bf 1}& 0& \cellcolor{gray!25}{\bf 1}& \cellcolor{orange!25}{\bf 1} &\cellcolor{gray!25}{\bf 1} &    0 \\
 \end{array}
 \right)
 \;\;\;\;\;\;\;\;
 A_2 =
 \left(
 \begin{array}{cccccc}
\cellcolor{blue!25}{\bf 1} & \cellcolor{blue!25}{\bf 1}& \cellcolor{purple!25}{\bf 1}& 0& 0& \cellcolor{purple!25}{\bf 1} \\
0 &\cellcolor{yellow!25}{\bf 1} &0 &\cellcolor{yellow!25}{\bf 1} &0 &\cellcolor{yellow!25}{\bf 1} \\
0 & 0 & \cellcolor{purple!25}{\bf 1} &0 &0 &\cellcolor{purple!25}{\bf 1} \\
0 & 0 &0 &\cellcolor{orange!25}{\bf 1}& \cellcolor{orange!25}{\bf 1}& 0 \\
\cellcolor{gray!25}{\bf 1} & 0& \cellcolor{gray!25}{\bf 1}& 0 & \cellcolor{gray!25}{\bf 1} &0 \\
\cellcolor{blue!25}{\bf 1} & \cellcolor{blue!25}{\bf 1}& 0 &\cellcolor{orange!25}{\bf 1} &\cellcolor{orange!25}{\bf 1}& 0 \\
 \end{array}
 \right)
 \;\;\;\;\;\;\;\;
 A_3 =  \left(
 \begin{array}{cccccc}
\cellcolor{blue!25}{\bf 1}& 1& 1& 0& 0& 0\\
0& \cellcolor{blue!25}{\bf 1}& 1 &1 &0& 0\\
0 &0& \cellcolor{blue!25}{\bf 1}& 1& 1& 0\\
0 &0& 0& \cellcolor{blue!25}{\bf 1}& 1& 1\\
1 &0& 0& 0& \cellcolor{blue!25}{\bf 1}& 1\\
0& 1& 0& 1& 0& 1\\
 \end{array}
 \right) $$

\caption{There are four representatives of basic $0,1$ matrices of real rank $4$ and size $6 \times 6$
that do not have the $2$-sum property: the matrix $C_6$ and the three  matrices $A_1,A_2,A_3$ presented in the figure.
Matrices $A_1, A_2$  have binary rank $5$ and matrices $A_3, C_6$ have binary rank $6$.
The first $5$ ones on the main diagonal of matrices $A_1, A_2, A_3$, are a maximal isolation set of size $5$,
whereas  $C_6$ has an isolation set of size $6$.}
\label{fig:rank4size6}	
\end{figure}

\begin{lemma}
\label{characterize6}
Let $M$ be a $0,1$ basic matrix of size $6 \times 6$ and $\Rreal(M) = 4$.
Then either $M$ has the $2$-sum property in which case $\Rbin(M),\Rbool(M) \leq 5$,
or $M$ is equivalent to $C_6$ or to one of the  $3$ matrices $A_1,A_2,A_3$ in Figure~\ref{fig:rank4size6}, and it holds:
\begin{itemize}
\item
If $M$ is one of $A_1, A_2$ then  $\Rbin(M) = \Rbool(M) = i(M) = 5$.
\item
If $M = A_3$ then $\Rbin(M) = \Rbool(M) = 6$  and  $i(M) = 5$.
\item
If $M = C_6$ then $\Rbin(M) = \Rbool(M) = i(M) = 6$.
\end{itemize}
\end{lemma}
\begin{proof}
The claim that any matrix $M$ which does not have the $2$-sum property is equivalent to $C_6$
or to one of the three matrices $A_1,A_2,A_3$ in Figure~\ref{fig:rank4size6},
is verified via a computer program which goes over all basic matrices of size $6 \times 6$ with real rank $4$,
while pruning the matrices examined as described in Section~\ref{IntroTechniques}.
The binary and Boolean rank of $A_1, A_2$ are $5$ as claimed, since each matrix has an isolation set of size $5$,
and its ones can be partitioned into $5$ monochromatic rectangles as shown in Figure~\ref{fig:rank4size6}.
By Lemma~\ref{Realkplus1Bin2k}, $\Rbin(C_6) = \Rbool(C_6) = i(C_6) = 6$ as claimed.

As for $A_3$, the first $5$ ones on the main diagonal of $A_3$ are an isolation set of size $5$,
and we now prove that it is not possible to partition the ones of $A_3$ into only $5$ monochromatic rectangles.
From this will follow that $\Rbin(A_3) = 6$.
Assume by contradiction that the ones of $A_3$ can be partitioned into five monochromatic rectangles $S_i$, $1 \leq i \leq 5$.
 Each of the first $5$ ones on the main diagonal should be in a different rectangle in this partition, since they are an isolation set.
The $1$ in position $(6,6)$ can only belong to $S_4$. Thus, the rectangles $S_i$ impose the following partial partition:
  $$A_3 =  \left(
 \begin{array}{cccccc}
\cellcolor{blue!25}{\bf 1}& 1& 1& 0& 0& 0\\
0& \cellcolor{yellow!25}{\bf 1}& 1 &1 &0& 0\\
0 &0& \cellcolor{purple!25}{\bf 1}& 1& 1& 0\\
0 &0& 0& \cellcolor{orange!25}{\bf 1}& 1& \cellcolor{orange!25}{\bf 1}\\
1 &0& 0& 0& \cellcolor{gray!25}{\bf 1}& \cellcolor{gray!25}{\bf 1}\\
0& 1& 0& \cellcolor{orange!25}{\bf 1}& 0& \cellcolor{orange!25}{\bf 1}\\
 \end{array}
 \right) $$
 But this means that the remaining uncovered $1$ on the fourth row cannot belong to any of the existing five rectangles and we get a contradiction.
 A similar argument shows that $\Rbool(M) = 6$ as well.

 Finally, we show that the maximal  isolation set in $A_3$ is of size $5$. Assume by contradiction that it has an isolation set $S$ of size $6$,
 and that the one in position $(6,6)$ is in $S$.
  Therefore, the only possibilities of a $1$ in the fifth row that can be added to the isolation set are either in position $(5,5)$ or in position $(5,1)$.
 If we add the $1$ in position $(5,5)$ to $S$ then there is no way to add a $1$ from the fourth row to the isolation set which does not belong to the same row or column as the
 two ones in $S$ or that does not belong to an all one $2\times 2$ sub-matrix, and we get a contradiction.
 Similarly, if we add to $S$ the $1$ in position $(5,1)$ then we must add to $S$ the following ones from rows $4$ through $2$:
  $$A_3 =  \left(
 \begin{array}{cccccc}
1& 1& 1& 0& 0& 0\\
0&  \cellcolor{blue!25}{\bf 1}& 1 &1 &0& 0\\
0 &0&  \cellcolor{blue!25}{\bf 1}& 1& 1& 0\\
0 &0& 0& 1& \cellcolor{blue!25}{\bf 1}& 1\\
\cellcolor{blue!25}{\bf 1} &0& 0& 0& 1 & 1\\
0& 1& 0& 1& 0& \cellcolor{blue!25}{\bf 1}\\
 \end{array}
 \right)
 $$
Now, since there is no way to add to $S$ a $1$ from the first row, we get a contradiction.

A similar contradiction is obtained if  the $1$ in position $(6,4)$ belongs to $S$.
Finally, assume that there exists an isolation set $S$ of size $6$ with the $1$ in position $(6,2)$.
Thus, $S$ must contain also the $1$ in position $(2,3)$ and this imposes the following structure on $S$ and, therefore, we cannot add any $1$ from the fifth row:
$$A_3 =  \left(
 \begin{array}{cccccc}
 \cellcolor{blue!25}{\bf 1}& 1& 1& 0& 0& 0\\
0& 1&  \cellcolor{blue!25}{\bf 1} &1 &0& 0\\
0 &0&  1& 1&  \cellcolor{blue!25}{\bf 1}& 0\\
0 &0& 0& 1& 1&  \cellcolor{blue!25}{\bf 1}\\
1 &0& 0& 0& 1 & 1\\
0& \cellcolor{blue!25}{\bf 1}& 0& 1& 0& 1 \\
 \end{array}
 \right)
 $$
 Hence, none of the $1$s in the sixth row can belong to an isolation set and we are done.
\end{proof}

\begin{corollary}
\label{C6}
The only  $0,1$ matrix $M$ of size $6 \times 6$ with $\Rreal(M)  = 4$ and $\Rbin(M) = \Rbool(M) = i(M) = 6$ is the circulant matrix $C_6$.
\end{corollary}

\begin{lemma}
Let $M$ be a $0,1$ basic matrix of size $6 \times 6$ and $\Rreal(M) = 4$.
If $M$ does not have the $2$-sum property then it has at least two pairs of rows or two pairs of columns whose sum is equal to the all one vector.
\end{lemma}
\begin{proof}
If $M$ does not have the $2$-sum property then it is equivalent to one of the $4$ representatives  in Lemma~\ref{characterize6}.
It is easy to verify that each of these $4$ representatives has the property claimed.
\end{proof}

Although there exists a basic $0,1$ matrix $M$ of size $7 \times 7$ such that $\Rreal(M) = 4$, we now show that the binary rank of each such matrix is at most $6$.

\begin{lemma}
\label{real47}
Let $M$ be a  basic $0,1$ matrix  of size $7 \times 7$ with $\Rreal(M) = 4$. Then $\Rbin(M) \leq 6$.
\end{lemma}
\begin{proof}
Any basic $0,1$ matrix $M$ of size $7 \times 7$ with $\Rreal(M) = 4$ is achieved by taking a basic $0,1$ matrix $A$ of size $6 \times 6$ with real rank $4$
and augmenting it with an additional column which is a linear
combination of the columns of $A$, and then adding a row which is a linear combination of the resulting matrix.
First note that if $A$ has two rows whose sum is another row in $A$ then the binary rank of $A$ is at most $5$.
By Lemma~\ref{sumofrowaugmented}, if we augment $A$ with a column which is a linear combination of the columns of $A$ then the resulting matrix has also two rows whose sum equals a third row.
Thus, the resulting matrix has binary rank at most $5$, and if we now augment it with an additional row we get binary rank at most $6$.
A similar argument holds if $A$ has two columns whose sum equals another column in $A$.

Therefore, assume now that $A$ does not have the $2$-sum property.
But then $A$ is equivalent to one of the four representatives in Lemma~\ref{characterize6}.
We ran a computer program which took each one of these representatives and tried to add to them a row and a column, such that the resulting $7 \times 7$ matrix is basic and has real rank $4$.
It turns out that each of the resulting matrices has the $2$-sum property, and thus their binary rank is at most $6$ as claimed.
\end{proof}

\begin{lemma}
\label{real478}
Let $M$ be a basic $0,1$ matrix with at most $7$ columns or at most $7$ rows, such that $\Rreal(M) = 4$. Then $\Rbin(M) \leq 6$.
\end{lemma}
\begin{proof}
By Lemma~\ref{real47} it is enough to consider basic $0,1$ matrices with $7$ columns and at least $8$ rows.
The same argument holds for matrices with $7$ rows and at least $8$ columns.
Consider first basic matrices of size $8 \times 7$ and real rank $4$.
The computer program we ran found only two representatives of such matrices:
$$
\left(
  \begin{array}{ccccccc}
 0 & 1 & 1 & 1 & 0 &0 & 1 \\
1 & 0 & 1 & 1 & 0 & 1 & 0 \\
1 & 1 & 0 & 1 & 1 & 0 & 0 \\
1 & 1 & 1 & 1 & 0 & 0 & 0 \\
0 & 0 & 1 & 1 & 0 & 1 & 1 \\
0 & 1 & 0 & 1 & 1 & 0 & 1 \\
1 & 0 & 0 & 1 & 1 & 1 & 0 \\
0 & 0 & 0 & 1 & 1 & 1 & 1 \\
  \end{array}
\right)
\;\;\;\;\;\;\;\;\;\;
\left(
  \begin{array}{ccccccc}
0 & 1 & 1 & 1 & 0 & 1 & 0 \\
1 & 0 & 1 & 1 & 1 & 0 & 0 \\
1 & 1 & 0 & 1 & 0 & 0 & 1 \\
1 & 1 & 1 & 1 & 1 & 0 & 1 \\
0 & 1 & 1 & 0 & 1 & 0 & 1 \\
1 & 0 & 0 & 1 & 0 & 0 & 0 \\
0 & 0 & 1 & 0 & 1 & 0 & 0 \\
0 & 1 & 0 & 0 & 0 & 0 & 1 \\
  \end{array}
\right)
$$
Both matrices have two columns whose sum equals a third column, and since there are seven columns the binary rank is at most $6$ by Lemma~\ref{sumoftworows}.
Now, any basic $0,1$ matrix with $7$ columns and more than $8$ rows with real rank $4$,
must include one of these two matrices as a sub-matrix, and any additional rows are some linear combination of the existing rows of these two matrices.
Thus, by Lemma~\ref{sumofrowaugmented} these augmented matrices also have two columns whose sum equals a third column and their binary rank is at most $6$ as required.
\end{proof}

\begin{lemma}
\label{real48}
Let $M$ be a basic $0,1$ matrix of size $8 \times 8$. Then $\Rreal(M) > 4$.
\end{lemma}
\begin{proof}
The lemma is proved using a computer program as described in Section~\ref{IntroTechniques},
which starts with all representatives of matrices of size $4 \times 4$ and real rank $4$,
and then augments them to get larger and larger matrices of real rank $4$.
While there are still a few matrices with at most $7$ rows or at most $7$ columns and real rank $4$, as described in Lemmas~\ref{real47} and~\ref{real478},
none of these matrices can be augmented to a basic $0,1$ matrix of size $8 \times 8$ and real rank $4$.
\end{proof}

\begin{theorem}
For any $0,1$ matrix $M$ with  $\Rreal(M) = 4$ it holds that $3 \leq \Rbool(M), i(M) \leq 6$ and $4 \leq \Rbin(M) \leq 6$.
\end{theorem}
\begin{proof}
It is enough to consider basic $0,1$ matrices.
%By Lemma~\ref{real48} there does not exist a $8 \times 8$ basic matrix with real rank $4$.
By  Lemma~\ref{basic_large} if there exists a basic matrix of real rank $4$ and size $n \times m$, where $n,m \geq 8$,
then there exists a basic sub-matrix of size $8 \times 8$ and real rank $4$, in contradiction to Lemma~\ref{real48}.
Hence, any basic $0,1$ matrix with real rank $4$ has at most $7$ rows or at most $7$ columns.
By Lemma~\ref{real47} the binary rank of any $7 \times 7$ basic matrix with real rank $4$ is at most $6$.
As to matrices with at most $7$ columns or rows, any such matrix has binary rank at most $6$ by Lemma~\ref{real478}.
Of course the binary rank is at least as large as the real rank.
Using Lemma~\ref{RealBool34}, we also have $3 \leq \Rbool(M), i(M) \leq \Rbin(M) \leq 6$.
\end{proof}

\bibliographystyle{plain}
\bibliography{realrank}

\end{document}